\documentclass[11 pt]{amsart}
\usepackage{graphicx} 
\usepackage{amsmath}
\usepackage{mathtools}
\usepackage
{amstext, amscd, setspace, tikz-cd,mathtools, enumerate, graphics, latexsym, siunitx}

\usepackage{pst-node}
\usepackage{tikz-cd} 

\usepackage{amssymb}
\usepackage{hyperref}

\usepackage{PTSerif} 

\usepackage[cmtip,all]{xy}
\newcommand{\properideal}{%
\mathrel{\ooalign{$\lneq$\cr\raise.22ex\hbox{$\lhd$}\cr}}}

\newcommand{\properring}{%
\mathrel{\ooalign{$\gneq$\cr\raise.22ex\hbox{$\rhd$}\cr}}}

\newcommand{\Q}{{\mathbb Q}}

\newcommand{\Z}{{\mathbb Z}}

\newtheorem{theorem}{Theorem}[section]
\newtheorem{corollary}{Corollary}[theorem]

\newtheorem{lemma}[theorem]{Lemma}
\newtheorem{proposition}[theorem]{Proposition}

\newtheorem{definition}[theorem]{Definition}
\newtheorem{example}[theorem]{Example}

\newtheorem{remark}{Remark}[theorem]

\numberwithin{equation}{section}

\catcode`\@=11
\catcode`\@=12

\theoremstyle{definition}
\numberwithin{equation}{subsubsection}

\title[On Variants of Inverse Cluster Size Problem \& General Magnification]{On Variants of Inverse Cluster Size Problem \& General Magnification}
\author{Shubham Jaiswal, M Krithika, P Vanchinathan}

\address{Department of Mathematics IIT Bombay, Powai, Mumbai 400 076, India.}

\email{sjaiswal@math.iitb.ac.in}

\smallskip

\address{School of Advanced Sciences, VIT Chennai,
Vandalur-Kelambakkam Road, Chennai 600 127, India}

\email{krithika.m2020@vitstudent.ac.in}

\bigskip

\bigskip

\email{vanchinathan@gmail.com}


\date{February 27, 2026.}


\begin{document}

\begin{abstract}
    In this article we establish certain variants of the Inverse Cluster Size problem. We introduce the notion of primitive extensions and establish the Primitive variant of the problem. Precisely, we prove the existence of primitive extensions over number fields of any given degree and cluster size less than the degree. We also introduce the notions of Strong and Weak General Magnification and the notion of general primitive extensions. We establish some interesting cases of the General primitive variant of the problem. We also recall the notion of totally real number fields and resolve the Totally real variant of the problem completely.
\end{abstract}

\maketitle

\section{Introduction}

 In Galois theory, a basic result is that for a finite separable extension $L/K$ the number of automorphisms of
$L$ fixing  $K$ is equal to the degree $[L:K]$
if and only if $L/K$ is a Galois extension. In general, we have the notion of cluster size of $L/K$, $r_K(L)$ which is the number of automorphisms of
$L$ fixing  $K$ (See \cite{Bhagwat_2025}, \cite{krithika2023root}, \cite{perlis2004roots}). It is not hard to see that the cluster size will actually be a divisor of the degree $[L:K]$.

\smallskip

After proving Lagrange's theorem in group theory, one naturally wonders about the converse and Cauchy's theorem and Sylow's theorem are such attempts at the converse. In the same spirit, a similar converse problem (which is called Inverse cluster size problem in \cite{Bhagwat_2025}) can be asked : Given a perfect field $K$ and integers $n>2$ and $r\mid n$, does there exist an extension $L/K$ of degree $n$ with cluster size $r$ ?\smallskip

The Inverse cluster size problem for $K=\mathbb{Q}$ was established in the pioneering work of Perlis \cite{perlis2004roots}, \cite{perlisroots}. The Inverse cluster size problem for number fields was resolved completely by the first author and Bhagwat in \cite{Bhagwat_2025} thus generalizing the result of Perlis for any number field. They used a similar approach using Shafarevich's theorem and Hilbertian field theory along with certain Galois-theoretic lemmas which extended the result from $\mathbb{Q}$ to any number field $K$. This generalization also improved on the generalization proved previously by the second and third author in \cite{krithika2023root}, which was proved by a process called Cluster magnification, reducing the problem to solving for the case $r=1$, but which excluded certain cases, namely $n = 2r$, where $r$ is odd for $\mathbb{Q}$ and $n = 2r$ for any number field $K\neq \mathbb{Q}$.

\smallskip

In this article we establish certain variants of the Inverse cluster size problem for number fields. In Section \ref{primitive}, the notion of an extension being primitive is introduced which means that the extension is not obtained from a nontrivial process of cluster magnification. We recall the notions of unique descending chains and unique ascending chains for extensions introduced in \cite{Bhagwat_2025}. We then establish a criterion in Theorem \ref{unique chain and primitive} for an extension to be primitive in terms of unique chains which serves as a major ingredient in the proof of our Theorem \ref{Prim inv clus} which is as follows.

\begin{theorem}
    (Primitive Inverse Cluster Size Problem for Number Fields) 
    
    Let $K$ be a number field. Let $n>2$ and $r|n$ and $r<n$. Then there exists a primitive extension $L/K$ of degree $n$ with cluster size $r_K(L)=r$.

\end{theorem}

In Section \ref{General}, the notions of Strong and Weak General Magnification are introduced and studied inspired from the notions of Strong and Weak Cluster Magnification in \cite{Bhagwat_2025} and motivated from results in \cite{krithika2025inflated}. In Section \ref{general primitive}, the notion of an extension being general primitive is introduced which means that the extension is not obtained from a nontrivial process of general magnification. We then go on to establish some interesting cases of the General Primitive Inverse Cluster Size Problem which is encapsulated in the following result.

\begin{theorem}\label{gen prim intro}

   Consider $(K,n,r)$ with $K$ being a number field, $n>2$ and $r|n$ and $r<n$. Then there exists a general primitive extension $L/K$ of degree $n$ with cluster size $r_K(L)=r$ in the following cases:

    \begin{enumerate}

\item $K$ is any number field.\smallskip

\begin{enumerate}

\item $(K,\ ^nP_k, k! )$ where $n>2$ and $1 \leq k \leq n-2$. In particular $(K, n, 1)$ where $n>2$.\smallskip

\item $(K, 2n, 2)$ where $n\geq 2$.
\smallskip

\item $(K, 4 {n\choose k}, 4)$ where $n>2$ and $1<k<n-1$.

\smallskip

\item $(K,pr,r)$ where $p$ is a prime and $r\ |\ p-1$ and $p-1>2r$ and either (i) $p\equiv 1\pmod 4$ or (ii) $p\equiv 3\pmod 4$ and $r$ is odd.

\smallskip

\item $(K,2n,n)$ where $n>2$ is odd.

\end{enumerate}
\smallskip

        \item $(\mathbb{Q},(p+1)r,r)$ where $p$ is a prime and $2r\ |\ p-1$ and $r\geq 3$.

    \end{enumerate}
\end{theorem}\smallskip

In Section \ref{tot real}, we recall the notion of totally real number fields and resolve the Totally Real Variant of the problem completely which is as follows.

\begin{theorem} (Totally Real Inverse Cluster Size Problem) 

Let $K$ be a totally real number field and $n>2$ and $r\mid n$. There exists an $L/K$ where $L$ is totally real such that $[L:K]=n$ and $r_K(L) = r$. 
\smallskip

Furthermore we show existence of general primitive $L/K$ (hence also primitive) where $L$ is totally real number field for the following cases of $(K,n,r)$ in Theorem \ref{gen prim intro}: Cases 1 (a), (b), (c) and Case $(K,2p,p)$ where $p>2$ is an odd prime.\smallskip



Additionally we show existence of primitive $L/K$ where $L$ is totally real number field whenever we have $1<r<n$ for $n$ odd.

\end{theorem}










\section{Primitive Inverse Cluster Size Problem}\label{primitive}

Let $K$ be a perfect field. We fix an algebraic closure $\bar{K}$ once and for all and work with finite extensions of $K$ contained in $\bar K$. Let $L/K$ be a finite extension and $\tilde{L}$ be its Galois closure inside $\bar{K}$. Let $G = {{\rm Gal}}(\tilde{L}/K)$ and $H= {{\rm Gal}}(\tilde{L}/L)$. We have the notion of cluster size of $L/K$, $r_K(L)$ which is the number of automorphisms of
$L$ fixing  $K$ which turns out to be the quantity $[N_G(H):H]$ and thus a divisor of the degree $[L:K]=[G:H]$ (See Section 2.1 in \cite{Bhagwat_2025}). We also have the notion of number of clusters of $L/K$, $s_K(L)$ which is $[G:N_G(H)]$ which is also the number of distinct fields inside $\bar{K}$ isomorphic to $L$ over $K$ (See Section 3.2 in \cite{Bhagwat_2025}).

\subsection{Primitive Extensions}

The following is Definition 4.1.1 in \cite{Bhagwat_2025}.

\begin{definition}

\label{SCM}
A finite extension $M/K$ is said to be obtained by strong cluster magnification from a subextension $L/K$ if we have the following:
\smallskip
 
 \begin{enumerate}
 \item  $[L:K] = n > 2,$ 
 \smallskip
      
\item there exists a finite Galois extension $F/K$ such that the Galois closure $\tilde{L}$ of $L$ in $\bar{K}$ and $F$ are linearly disjoint over $K$. \smallskip

\item $LF=M$.\smallskip

 \end{enumerate}

The quantity $[F:K]$ is called the magnification factor and denoted by $d$. The magnification is called trivial if $d=1$ and nontrivial otherwise.

\end{definition}

The following is reformulation of Definition 4.2.4 in \cite{Bhagwat_2025} for field extensions (This notion for polynomials occurs in \cite{krithika2023root} as well).

\begin{definition}

    An extension $L/K$ is called primitive if it is not obtained by a nontrivial strong cluster magnification from any subextension over $K$. 
\end{definition}

\begin{proposition} \label{simple prop}
Consider $L/K$. Let $G={{\rm Gal}}(\tilde{L}/K)$ and $H={{\rm Gal}}(\tilde{L}/L)$.
     If $H$ is not contained in any proper normal subgroup of $G$ then $L/K$ is primitive. In particular, if $G$ is simple then $L/K$ is primitive. \end{proposition}
    
  \begin{proof} Suppose $H$ is not contained in any proper normal subgroup of $G$. Assume $L/K$ is not primitive. Hence $L/K$ is obtained by nontrivial strong cluster magnification from some subextension say $L'/K$ through $F/K$ as in Definition \ref{SCM}. Thus $F=\tilde{L}^{H'}$ where $H\subset H'\unlhd G$. Thus $H'=G$. Hence $F=K$ which is a contradiction.\end{proof}

\subsection{Unique Chains for Extensions and Primitive Extensions}

We have the notions of unique descending chains and unique ascending chains for extensions introduced in Section 7 in \cite{Bhagwat_2025}. We recall those notions here. Let $L/K$ be a nontrivial finite extension. \smallskip

There is a unique strictly descending chain of subextensions
       \[ L=N_0\supsetneq N_1 \supsetneq N_2 \supsetneq \dots \supsetneq N_k\]
        such that for all $1\leq i\leq k$, $N_i$ is the unique intermediate extension for $N_{i-1}/K$ such that $N_{i-1}/N_i$ is Galois of maximum possible degree. 
\smallskip

        There is a unique strictly ascending chain of subextensions  \[ K=F_0\subsetneq F_1 \subsetneq F_2 \subsetneq \dots \subsetneq F_l \] such that for all $1\leq j\leq l$, $F_j$ is the unique intermediate extension for $L/F_{j-1}$ such that $F_j/F_{j-1}$ is Galois of maximum possible degree.\smallskip

Both the unique chains terminate since $L/K$ is finite. The following establishes a criterion for an extension to be primitive in terms of unique chains.

\begin{theorem} \label{unique chain and primitive}
    Consider a nontrivial finite extension $L/K$. Suppose one of the fields in the unique descending chain of $L/K$ which is different from both $L$ and $K$, coincides with one of the fields in the unique ascending chain of $L/K$. Then $L/K$ is primitive.
\end{theorem}

\begin{proof}
    We have $N_i=F_j\neq L,K$ for some $i$ and some $j$. Assume that $L/K$ is not primitive. Hence $L/K$ is obtained by nontrivial strong cluster magnification from some subextension say $L'/K$ through $F/K$ as in Definition \ref{SCM}. Now $F/K$ is Galois and we have that $F_1$ is the unique intermediate extension for $L/K$ such that $F_1/K$ is Galois of maximum possible degree. We have $FF_1/K$ is also Galois. Hence $FF_1=F_1$. Thus $F\subset F_1$. Hence $F\subset F_j=N_i$.
    \smallskip
    
     Let $\tilde{L}$ be Galois closure of $L/K$ and $G={{\rm Gal}}(\tilde{L}/K)$ and $H={{\rm Gal}}(\tilde{L}/L)$. By Theorem 7.1.1 in \cite{Bhagwat_2025}, $N_1=\tilde{L}^{N_G(H)}$. Now since $L=L'F$ and $\tilde{L'}\cap F=K$, we have that $L/L'$ is Galois. Hence $L'=\tilde{L}^{H'}$ where $H\unlhd H' \subset G$. Thus $H'\subset N_G(H)$. Hence $N_1\subset L'$. Thus $N_i\subset L'$ and hence $F\subset L'$. This forces $F=K$ which gives a contradiction to the magnification being nontrivial.
\end{proof}

\begin{remark}
     A shorter proof of above theorem by using results from Section 8.4 in \cite{Bhagwat_2025}: We have $N_i=F_j\neq L,K$. Suppose $L/K$ is obtained by nontrivial strong cluster magnification from some subextension say $L'/K$ through $F/K$. By Theorem 8.4.1 and Remark 8.4.2 in \cite{Bhagwat_2025}, $N_i\subset L'$. By Theorem 8.4.3 and Remark 8.4.4 in \cite{Bhagwat_2025}, $F\subset F_j$. Thus $F\subset L'$ which gives a contradiction.
\end{remark}

\begin{remark}\label{galois not prim}
The assumption $N_i=F_j\neq L,K$ is important in the theorem. Suppose $L/K$ is Galois with Galois group $A\times B$ for $A$ and $B$ nontrivial groups and $|A|>2$. The unique descending chain for $L/K$ is $L=N_0\supsetneq K=N_1$ and unique ascending chain for $L/K$ is $K=F_0\subsetneq L=F_1$. Thus $N_0=F_1=L$ and $N_1=F_0=K$. But by Corollary 4.1.8 in \cite{Bhagwat_2025} we have that $L/K$ is not primitive.

\end{remark}

\begin{remark}\label{Snremark}
    The converse of Theorem \ref{unique chain and primitive} is false. Let $f$ over $K$ be irreducible of deg $n>2$ with Galois group ${\mathfrak S}_n$ with roots $\alpha_i\in \bar{K}$ for $1\leq i\leq n$. For $1 \leq k \leq n-2$, let $L_k=K(\alpha_1,\dots, \alpha_k)$. Assume $k\neq 1, n/2$. By Theorem 7.3.6 in \cite{Bhagwat_2025}, the unique descending chain for $L_k/K$ is $L_k\supsetneq N_k$ and unique ascending chain for $L_k/K$ is $K$. Hence $L_k/K$ doesn't satisfy hypothesis of Theorem \ref{unique chain and primitive} but we have that $L_k/K$ is primitive. This is because we don't have any $K\subsetneq F\subset L_k$ such that $F/K$ is Galois.
\end{remark}

\begin{example}
    Consider $K=\mathbb{Q}$ and let $n=2^k$ for $k\geq 2$. Let $c$ be a positive odd integer such that $f=x^n-c$ is
  an irreducible polynomial over $\mathbb{Q}$. Let $a = c^{1/n}$ be the positive real root of $f$. Let $L=\mathbb{Q}(a)$. Then $L/K$ is primitive.\smallskip

Since $c$ is odd, it does not divide the discriminant of $\mathbb{Q}(\zeta)/\mathbb{Q}$ where $\zeta$ is a primitive $n$-th root 
of unity in $\bar{\mathbb{Q}}$. Thus $\sqrt{c}\not \in \mathbb{Q}(\zeta)$. By Proposition 7.3.1 in \cite{Bhagwat_2025}, the unique descending chain is $L=N_0\supsetneq N_1\supsetneq \dots \supsetneq N_{k}=K$ and the unique ascending chain is $K=F_0\subsetneq F_1\subsetneq \dots \subsetneq F_{k}=L$ with $N_i= F_{k-i}=\mathbb{Q}(a^{2^i})$ for all $0\leq i\leq k$. Since $k\geq 2$, we are done by Theorem \ref{unique chain and primitive}.

    \end{example}

\subsection{Primitive Inverse Cluster Size Problem}

\begin{theorem}\label{Prim inv clus}
    (Primitive Inverse Cluster Size Problem for Number Fields) Let $K$ be a number field. Let $n>2$ and $r|n$ and $r<n$. Then there exists a primitive extension $L/K$ of degree $n$ with cluster size $r_K(L)=r$.

\end{theorem} 

\begin{proof}

We proceed in the same way as the first author and Bhagwat proceeded in the proof of Inverse Cluster Size Problem for Number Fields Theorem 3.1.1 in \cite{Bhagwat_2025}. Let $s=n/r$.\smallskip

For $r=1$ we have that there exists $L/K$ of degree $n$ with ${{\rm Gal}}(\tilde{L}/K)$ as ${\mathfrak S}_n$. This $L/K$ satisfies $r_K(L)=1$. The subgroup of ${\mathfrak S}_n$ fixing $L$ is isomorphic to ${\mathfrak S}_{n-1}$. Since ${\mathfrak S}_{n-1}$ is not contained in any proper normal subgroup of ${\mathfrak S}_{n}$. Hence by Proposition \ref{simple prop}, $L/K$ is primitive.  

\smallskip

For $r>1$ we have that there exists $L/K$ with $G={{\rm Gal}}(\tilde{L}/K)$ solvable such that action of the group is transitive on $n$ points, and a point stabiliser fixes precisely $r$ points. We have $G=(\Z/r\Z)^s \rtimes \Z / s\Z$ with semidirect product group law \[ ((a_1,\dots, a_s),b)\cdot((c_1,\dots,c_s),d)=((a_1,\dots, a_s) + (b\cdot (c_1,\dots, c_s)), b+d), \]
  where $b\cdot(c_1,\dots, c_s)=(c_{b+1},\dots, c_s,c_1,\dots, c_{b})$ for $b\neq 0$ \&  $0\cdot(c_1,\dots, c_s)=(c_1,\dots, c_s)$. Any point stabiliser is isomorphic to $(\Z/r\Z)^{s-1}$. We have $[L:K]=n$ and $r_K(L)=r$. \smallskip
  
Since $1<r<n$. Thus by Theorem 7.3.4 in \cite{Bhagwat_2025}, the unique descending chain is $L\supsetneq N\supsetneq K$ and the unique ascending chain is $K\subsetneq F \subsetneq L$ and they both coincide. Hence $N=F\neq L,K$. Thus by Theorem \ref{unique chain and primitive}, $L/K$ is primitive.
\end{proof}

\begin{remark}\label{Galois case remark}
    The same proof works for the case $r=n=p^k>2$ where $p$ is a prime. In that case $L/K$ is Galois with ${{\rm Gal}}(L/K)$ as $\mathbb{Z}/p^k \mathbb{Z}$. By uniqueness in structure theorem for finite abelian groups, ${{\rm Gal}}(L/K)$ is not isomorphic to $A\times B$ for any nontrivial groups $A$ and $B$. Hence by Corollary 4.1.8 in \cite{Bhagwat_2025}, $L/K$ is primitive.

    \smallskip
    
    But the same proof doesn't work for the case $r=n$ when $n$ has more than one prime in its prime factorization. In that case $L/K$ is Galois with ${{\rm Gal}}(L/K)$ as $\mathbb{Z}/n\mathbb{Z}$. By chinese remainder theorem, ${{\rm Gal}}(L/K)$ is isomorphic to $A\times B$ for some nontrivial groups $A$ and $B$ with $|A|>2$. Hence $L/K$ is not primitive similarly as in Remark \ref{galois not prim}.\smallskip

    In fact we can show that there are integers $n>2$ such that any extension $L/K$ which is Galois of degree $n$, is not primitive. For example, for $n=pq$ where $p<q$ are distinct primes and $p\nmid q-1$, the Galois group of $L/K$ is cyclic and hence by previous paragraph, $L/K$ is not primitive.\end{remark}
\smallskip

\subsection{Primitive Inverse Ascending Index Problem}
\hfill
\smallskip

 For an extension $L/K$, we have the notion of ascending index of $L/K$, $t_K(L)$ which is $[G:H^G]$ (where $H^G$ is normal closure of $H$ in $G$) and the quantity $u_K(L)$ which is $[H^G:H]$ (See Section 7.2 and Section 9 in \cite{Bhagwat_2025} for basic properties of ascending index of field extension).\smallskip

By proof of Theorem \ref{Prim inv clus} and by the proof of Inverse Ascending Index Problem for Number Fields Theorem 9.0.5 in \cite{Bhagwat_2025}, we have the following.

\begin{theorem}(Primitive Inverse Ascending Index Problem for Number Fields) Let $K$ be a number field. Let $n>2$ and $t|n$ and $t<n$. Then there exists a primitive extension $L/K$ of degree $n$ with ascending index $t_K(L)=t$.
\end{theorem}

\section{General Magnification}\label{General}

Just like the Cluster Magnification Theorem (Theorem 1 in \cite{krithika2023root}) proved by the second and the third author inspired the first author and Bhagwat to define the concept of Strong Cluster Magnification and Weak Cluster Magnification in Section 4.1 and Section 4.3 in \cite{Bhagwat_2025} respectively. In the same way, inspired from the Lemma 2 and the proof of the Main Theorem in \cite{krithika2025inflated} by the second and the third author, we define in this section a more general notion of magnification. Then cluster magnification would turn out to be a special case of that.

\subsection{Strong General Magnification}

\begin{definition}\label{SGM}
    A finite extension $M/K$ is said to be obtained by strong general magnification from a subextension $L/K$ if we have the following:
\smallskip
 
 \begin{enumerate}
 \item  $[L:K] > 1,$ 
 \smallskip
      
\item there exists a finite extension $J/K$ such that the Galois closure $\tilde{L}$ of $L$ in $\bar{K}$ and Galois closure $\tilde{J}$ of $J$ in $\bar{K}$ are linearly disjoint over $K$ (which means $\tilde{L}\cap \tilde{J}=K$). \smallskip

\item $LJ=M$.\smallskip

 \end{enumerate}

 The tuple $(r_K(J),s_K(J),t_K(J),u_K(J))$ is called the magnification tuple denoted by $(r,s,t,u)$. 
The magnification is called trivial if $J = K$ and nontrivial otherwise. 
 \end{definition}

 \begin{remark}
     Definition \ref{SGM} is symmetric in the nontrivial case. Suppose $M/K$ is obtained by nontrivial strong general magnification from $L/K$ through $J/K$. Then $M/K$ is obtained by nontrivial strong general magnification from $J/K$ through $L/K$.
 \end{remark}

\begin{remark}
    Strong General Magnification for magnification tuple $(d,1,d,1)$ is precisely Strong Cluster Magnification for magnification factor $d$.

\end{remark}

Let $\tilde{M}$ be Galois closure of $M/K$ inside $\bar{K}$. Let $G'= {{\rm Gal}}(\tilde{M}/K)$ and $H' = {{\rm Gal}}(\tilde{M}/M)$.

\begin{proposition}
    
\label{SGM prop}
    Suppose $M/K$ is obtained by strong general magnification from  $L/K$ through $J/K$. Let $\tilde L$ and $\tilde{J}$ be as in the Definition \ref{SGM}. Let $G = {{\rm Gal}}(\tilde{L}/K)$, $H= {{\rm Gal}}(\tilde{L}/L)$, $S= {{\rm Gal}}(\tilde{J}/K)$ and $T = {{\rm Gal}}(\tilde{J}/J)$. Then the following hold.

\smallskip 

\begin{enumerate}

\smallskip  \item $\tilde{M}=\tilde{L} \tilde{J}$ and $M\cap \tilde{L}=L$ and $M\cap \tilde{J}=J$ and $M\tilde{L}\cap M\tilde{J}=M$.

\smallskip \item $G'\cong G\times S$ where isomorphism is given by  $\lambda\in G' \mapsto (\lambda|_{\tilde{L}},\lambda|_{\tilde{J}})$.

\smallskip

\item  Furthermore $H'\cong H\times T \subset G\times S$ under the above isomorphism.

\smallskip \item $J$ is uniquely determined by $L$ and $M$. 

\smallskip \item If $M/K$ is Galois then $L/K$ and $J/K$ too are Galois and $M/K$ is obtained by strong cluster magnification from $L/K$ through $J/K$.

    \end{enumerate}
\end{proposition}

The following proposition justifies the terminology in Definition \ref{SGM}. The proof follows from the proof of Lemma 6.5 in \cite{jaiswal2025rootcapacityintersectionindicium} but we still provide it for the sake of completeness.

\begin{proposition}
    
 \label{GM}(General Magnification) Let $M/K$ be obtained by strong general magnification from $L/K$ with magnification tuple $(r,s,t,u)$. Then $[M:K]=r\cdot s\cdot  [L:K]=t\cdot u\cdot  [L:K]$ and $r_K(M)=r\cdot r_K(L)$ and $s_K(M)=s\cdot s_K(L)$ and $t_K(M)=t\cdot t_K(L)$ and $u_K(M)=u\cdot u_K(L)$.
    \end{proposition}

\begin{proof}
 By identifying the groups with their respective images under the isomorphism in Proposition \ref{SGM prop} (2), we have $G'= G\times S$, $H'= H\times T$. Thus $$[M:K]=[L:K][J:K]= r_K(J) s_K(J) [L:K]= t_K(J)u_K(J)\ [L:K].$$ Also $N_{G'}(H')= N_G(H)\times N_S(T)$. Hence $r_K(M)=[N_{G'}(H'):H']=r_K(L) r_K(J)$. Since\\
 $r_K(M) s_K(M)=[M:K]$. Thus $s_K(M)=s_K(J)\ s_K(L)$. \smallskip

  We can also show that $H'^{G'}= H^G \times T^S$ where $H^G$ denotes the normal closure of $H$ in $G$. Hence $t_K(M)=$ $[G': H'^{G'}]=t_K(L) t_K(J)$. Since $t_K(M) u_K(M)=[M:K]$. Thus $u_K(M)=u_K(J) u_K(L)$.
\end{proof}

We have the following hereditary property for strong general magnification.

\begin{proposition} \label{hereditary}
    Suppose $M/K$ is obtained by strong general magnification from  $L/K$ through $J/K$ as in Def \ref{SGM}. Let $K\subset L' \subset L$ and $K\subset J'\subset J$ with $[L:L']>1$. Then $M/L'J'$ is obtained by strong general magnification from $LJ'/L'J'$ through $L'J/L'J'$ with magnification tuple\\
    $(r_{J'}(J),s_{J'}(J),t_{J'}(J),u_{J'}(J))$.
    
\end{proposition}

\begin{proof}
 We check that the conditions in Definition \ref{SGM} hold. Let notation be as in Proposition \ref{SGM prop}. Let $L'=\tilde{L}^{H_0}$ and $J'=\tilde{J}^{T_0}$. Thus $L' J'=\tilde{M}^{H_0\times T_0}$. Let $L_1$ be Galois closure of $LJ'/L'J'$ and $J_1$ be Galois closure of $L'J/L'J'$. 
 
 \smallskip

Since $\tilde{L}\cap \tilde{J}=K$. Hence $\tilde{L}\cap J'=K$. Thus by Lemma 2.2.2 in \cite{Bhagwat_2025} we have $\tilde{L}\cap L'J'=L'$. Hence by Theorem 2.6 in \cite{conrad2023galois}, we have ${{\rm Gal}}(\tilde{L}J'/LJ')\cong {{\rm Gal}}(\tilde{L}/L)=H$ and ${{\rm Gal}}(\tilde{L}J'/L'J')\cong {{\rm Gal}}(\tilde{L}/L')=H_0$. Thus $[LJ':L'J']>1$ and $L_1\subset \tilde{L}J'$. Similarly we have $L'\cap \tilde{J}=K$ and hence $L'J'\cap \tilde{J}=J'$ and ${{\rm Gal}}(L'\tilde{J}/L' J)\cong {{\rm Gal}}(\tilde{J}/J)=T$ and ${{\rm Gal}}(L'\tilde{J}/L'J')\cong {{\rm Gal}}(\tilde{J}/J')=T_0$ and $J_1\subset L'\tilde{J}$.

\smallskip

Now $(\tilde{L}J')(L'\tilde{J})=\tilde{L}\tilde{J}=\tilde{M}$ and ${{\rm Gal}}(\tilde{M}/L'J')=H_0\times T_0$. By Theorem 2.1 in \cite{conrad2023galois} we have $\tilde{L}J'\cap L'\tilde{J}=L'J'$. Hence $L_1\cap J_1=L'J'$. Also $(LJ')(L'J)=LJ=M$. Finally since the Galois groups are isomorphic, we have 
that magnification tuple $(r_{L'J'}(L'J),s_{L'J'}(L'J),t_{L'J'}(L'J),u_{L'J'}(L'J))=(r_{J'}(J),s_{J'}(J),t_{J'}(J),u_{J'}(J))$.\end{proof}

By letting $J'=K$, we have the following.

\begin{corollary}
    Let $M/K$ be obtained by strong general magnification from $L/K$ through $J/K$. Then for any $K\subset L'\subset L$ with $[L:L']>1$, the extension $M/L'$ is obtained by strong general magnification from $L/L'$ through $L'J/L'$ with same magnification tuple.
\end{corollary}

    
    
\smallskip

   We have an equivalent criterion for strong general magnification for a field extension in terms of Galois groups.\smallskip

\begin{proposition}

 \label{SGM criterion}
    An extension $M/K$ is obtained by nontrivial strong general magnification from some subextension $L/K$
    if and only if ${{\rm Gal}}(\tilde{M}/K)\cong A\times B$ for nontrivial groups $A$ and $B$ and ${{\rm Gal}}(\tilde{M}/M)\cong A'\times B'$ (under the same isomorphism) for a subgroup $A'\subset A$ with $[A:A']>1$ and a subgroup $B'\subset B$ with $[B:B']>1$.\smallskip

\end{proposition}

\subsection{Weak General Magnification}

The following is Definition 4.3.1 in \cite{Bhagwat_2025}.

\begin{definition}
        $M/K$ is said to be obtained by weak cluster magnification from a subextension $L/K$ if $r_K(L)| r_K(M)$. We call $d=r_K(M)/r_K(L)$ as the magnification factor. The magnification is called trivial if $d=1$ and nontrivial otherwise.
\end{definition}

Now we define the following.

\begin{definition}

    $M/K$ is said to be obtained by weak general magnification from a subextension $L/K$ if $r_K(L)|r_K(M)$ and $s_K(L)|s_K(M)$ and $t_K(L)|t_K(M)$ and $u_K(L)|u_K(M)$. We call $(r,s,t,u)=(r_K(M)/r_K(L),s_K(M)/s_K(L),t_K(M)/t_K(L),u_K(M)/u_K(L))$ as the magnification tuple. The magnification is called trivial if $r=s=t=u=1$ and nontrivial otherwise.
\end{definition}

\begin{remark}
    Clearly Weak General Magnification implies Weak Cluster Magnification. But the converse is false. Consider the case in Remark \ref{Snremark}. Consider $1\leq j<k\leq n-2$. By proof of Theorem 3 in \cite{krithika2023root}, we have $r_K(L_k)=k!$ and $r_K(L_j)=j!$. Thus $s_K(L_k)=\ {n\choose k}$ and $s_K(L_j)=\ {n\choose j}$. Thus $L_k/K$ is obtained by nontrivial weak cluster magnification from $L_j/K$ but is not in general obtained from non trivial weak general magnification from $L_j/K$. For example ${4\choose 1} \nmid\ {4\choose 2}$.

\end{remark}

\begin{remark}

    From Definition \ref{SGM} and Proposition \ref{GM}, if $M/K$ is obtained by strong general magnification from $L/K$ then $M/K$ is obtained by weak general magnification from $L/K$.\smallskip

    Consider the case in above Remark for $k > n/2 \ \text{and}\ j=n-k$. Then $s_K(L_k)=s_K(L_j)$. Also $t_K(L_k)=t_K(L_j)=1$ by Theorem 7.3.6 in \cite{Bhagwat_2025}. Hence we also have that $u_K(L_j)=[L_j:K]$ divides $[L_k:K]=u_K(L_k)$. Hence $L_k/K$ is obtained by nontrivial weak general magnification from $L_j/K$ but is not obtained by strong general magnification from $L_j/K$ since $\tilde{L_{j}}\cap L_k=L_k\neq L_j$ (which contradicts Proposition \ref{SGM prop} (1)).\smallskip
    
    This justifies the word `weak' in above definition.\end{remark}
    
    \smallskip

\subsection{Strong General Magnification and Unique Chains}



  

\begin{theorem}\label{unique chain SGM}
     Suppose $M/K$ is obtained by strong general magnification from  $L/K$ through $J/K$. 
     
     \begin{enumerate}
\item Let the unique descending chain for $L/K$ be $L=N_0\supsetneq N_1 \supsetneq N_2 \supsetneq \dots \supsetneq N_k$ and the unique descending chain for $J/K$ be $J=N'_0\supsetneq N'_1 \supsetneq N'_2 \supsetneq \dots \supsetneq N'_{k'}$. Then the unique descending chain for $M/K$ is $M=N_0 N'_0\supsetneq N_1 N'_1 \supsetneq N_2 N'_2 \supsetneq \dots \supsetneq N_{k'} N'_{k'}\supsetneq N_{k'+1} N'_{k'}\supsetneq \dots  \supsetneq N_{k} N'_{k'}$ for $k\geq k'$ and $M=N_0 N'_0\supsetneq N_1 N'_1 \supsetneq N_2 N'_2 \supsetneq \dots \supsetneq N_{k} N'_{k}\supsetneq N_{k} N'_{k+1}\supsetneq \dots  \supsetneq N_{k} N'_{k'}$ for $k<k'$.

\smallskip

\item Let the unique ascending chain for $L/K$ be $K=F_0\subsetneq F_1 \subsetneq F_2 \subsetneq \dots \subsetneq F_l$ and the unique ascending chain for $J/K$ be $K=F'_0\subsetneq F'_1 \subsetneq F'_2 \subsetneq \dots \subsetneq F'_{l'}$. Then the unique ascending chain for $M/K$ is $K=F_0 F'_0\subsetneq F_1 F'_1 \subsetneq F_2 F'_2 \subsetneq \dots \subsetneq F_{l'} F'_{l'}\subsetneq F_{l'+1} F'_{l'}\subsetneq \dots  \subsetneq F_{l} F'_{l'}$ for $l\geq l'$ and $K=F_0 F'_0\subsetneq F_1 F'_1 \subsetneq F_2 F'_2 \subsetneq \dots \subsetneq F_{l} F'_{l}\subsetneq F_{l} F'_{l+1}\subsetneq \dots  \subsetneq F_{l} F'_{l'}$ for $l< l'$.
         
     \end{enumerate}
     
     \end{theorem}

\begin{proof} Let notation be as in Proposition \ref{SGM prop}.
    \begin{enumerate}
        \item By Theorem 7.1.1 (1) in \cite{Bhagwat_2025}, unique intermediate extension in unique descending chain corresponds to subextension fixed by the normalizer subgroup. Now by proof of Proposition \ref{GM}, $N_{G'}(H')= N_G(H)\times N_S(T)$. Since $\tilde{L}^{N_G(H)}=N_1$ and $\tilde{J}^{N_S(T)}=N'_1$. Thus $\tilde{M}^{N_{G'}(H')}=N_1 N'_1$. \smallskip
        
        Now clearly Galois closures of $N_1/K$ and $N'_1/K$ are linearly disjoint over $K$ since Galois closures of $L/K$ and $J/K$ are linearly disjoint over $K$. Thus $N_1 N'_1/K$ is obtained by strong general magnification from $N_1/K$ through $N'_1/K$. Thus proceeding in the same way as in the earlier paragraph, we are done. 

        \smallskip

        \item By Theorem 7.2.1 (3) in \cite{Bhagwat_2025}, unique intermediate extension in unique ascending chain corresponds to subextension fixed by the normal closure subgroup. Now by proof of Proposition \ref{GM}, $H'^{G'}= H^G \times T^S$. Since $\tilde{L}^{H^G}=F_1$ and $\tilde{J}^{T^S}=F'_1$. Thus $\tilde{M}^{H'^{G'}}=F_1 F'_1$.\smallskip

By Proposition \ref{hereditary}, $M/F_1F'_1$ is obtained by strong general magnification from $LF'_1/F_1F'_1$ through $F_1J/F_1F'_1$. By the proof of Proposition \ref{hereditary}, the Galois groups are isomorphic which implies that unique intermediate extension in unique ascending chain for $LF'_1/F_1F'_1$ is $F_2 F'_1/F_1$ and for $F_1J/F_1F'_1$ is $F_1F'_2/F_1F'_1$. Hence by previous paragraph, unique intermediate extension in unique ascending chain for $M/F_1F'_1$ is $(F_2F'_1)(F_1F'_2)=F_2F'_2$. Proceeding in the same way, we are done.\end{enumerate}\end{proof}

\begin{remark}\label{ease notation}
    For ease of notation in statement of Theorem \ref{unique chain SGM}, for $k\geq k'$, one can define $N'_{k'+1}=N'_{k'+2}=\dots =N'_k=N'_{k'}$. Similarly one can define for $k<k',\ l\geq l',\ l<l'$. Thus the field in unique descending chain of $M/K$ at $i$-th step (where $i\leq max\{k,k'\}$) is simply $N_iN'_i$. Similarly we have for unique ascending chain.  
\end{remark}

\begin{theorem}
    \label{uniq chain coinc GM}

    Consider a nontrivial extension $M/K$ and suppose $M/K$ is obtained by strong general magnification from  $L/K$ through $J/K$. Further assume that one of the fields in the unique descending chain of $M/K$ which is different from both $M$ and $K$, coincides with one of the fields in the unique ascending chain of $M/K$. Then atleast one of the following is true.

    \begin{enumerate}
        \item $L/K$ is primitive.

        \item The magnification is nontrivial and $J/K$ is primitive.

        \item The magnification is non trivial and $r_K(J)=1$ and $t_K(L)=1$ and unique ascending chain of $J/K$ terminates at $J$ and unique descending chain of $L/K$ terminates at $K$.

        \item The magnification is non trivial and $r_K(L)=1$ and $t_K(J)=1$ and unique ascending chain of $L/K$ terminates at $L$ and unique descending chain of $J/K$ terminates at $K$.
\end{enumerate}
\end{theorem}

\begin{proof}
    By Theorem \ref{unique chain SGM}, we know the structure of unique chains of $M/K$ in terms of unique chains of $L/K$ and $J/K$. Following the notation in Remark \ref{ease notation}, we have $N_iN'_i=F_jF'_j\neq M, K$ for some $i$ and some $j$. Now because of $\tilde{L}\cap \tilde{J}=K$, we have $N_i=\tilde{L}\cap N_i N'_i=\tilde{L}\cap F_jF'_j=F_j$ and similarly have $N'_i=F'_j$. We also have that $N_iN'_i\neq M$ is equivalent to $N_i\neq L$ or $N'_i\neq J$. Also $N_iN'_i\neq K$ is equivalent to $N_i\neq K$ or $N'_i\neq K$. Thus atleast one of the following cases must be true.\smallskip

    \begin{enumerate}

\item We have $N_i\neq L,K$. Since $N_i=F_j$. Thus by Theorem \ref{unique chain and primitive}, $L/K$ is primitive.\smallskip

\item We have $N'_i\neq J,K$. Hence $J\neq K$. Since $N'_i=F'_j$. Thus by Theroem \ref{unique chain and primitive}, $J/K$ is primitive.

\smallskip

\item We have $N_i\neq L$ but $N'_i=J$ and we have $N'_i\neq K$ but $N_i=K$. Now clearly $J=N'_i\neq K$. Thus the magnification is nontrivial. Also $i>0$ as $N_i\neq L$. Since $N'_i=J$ for $i>0$, we have that unique descending chain of $J/K$ is singleton $J=N'_0$. Hence by Theorem 7.1.1 in \cite{Bhagwat_2025}, we get $r_K(J)=1$. Now since $N_i=K$, so the unique descending chain of $L/K$ terminates at $K$. Now $N_i=F_j$ and $N'_i=F'_j$. Thus we have $F_j\neq L$ but $F'_j=J$ and we have $F'_j\neq K$ but $F_j=K$. Now $j>0$ as $F'_j\neq K$. Since $F_j=K$ for $j>0$, we have that unique ascending chain of $L/K$ is singleton $K=F_0$. Hence by Theorem 7.2.1 in \cite{Bhagwat_2025}, we get $t_K(L)=1$. Now since $F'_j=J$, so the unique ascending chain of $J/K$ terminates at $J$.

\smallskip

\item We have $N'_i\neq J$ but $N_i=L$ and we have $N_i\neq K$ but $N'_i=K$. This case can be dealt similarly as the previous case.\end{enumerate}\end{proof}

\section{General Primitive Inverse Cluster Size Problem}
\label{general primitive}

\subsection{General Primitive Extensions}
We define the following notion.

\begin{definition}
    An extension $L/K$ is called general primitive if it is not obtained by a nontrivial strong general magnification from any subextension over $K$.
\end{definition}

\begin{remark}
    Clearly General Primitive implies Primitive. But the converse is false.  By results in \cite{volklein1996groups}, we have alternating group ${\mathfrak A}_{n}$ for $n\geq 5$ to be realizable as a Galois group for infinitely many pairwise linearly disjoint Galois extensions over $\mathbb{Q}$. Thus by Lemma \ref{perm}, we have degree $n$ extensions $L/\mathbb{Q}$ and $J/\mathbb{Q}$ such that their Galois closures are linearly disjoint over $\mathbb{Q}$ that is $\tilde{L}\cap \tilde{J}=\mathbb{Q}$ and ${{\rm Gal}}(\tilde{L}/\mathbb{Q})$ and ${{\rm Gal}}(\tilde{J}/\mathbb{Q})$ are isomorphic to $\mathfrak A_{n}$. Thus $LJ/\mathbb{Q}$ is obtained by nontrivial strong general magnification from $L/\mathbb{\mathbb{Q}}$ through $J/\mathbb{Q}$. Hence $LJ/\mathbb{Q}$ is not general primitive. \smallskip

    We claim that $LJ/\mathbb{Q}$ is primitive. We have that Galois closure of $LJ/\mathbb{Q}$ is $\tilde{L}\tilde{J}$. We also have that ${{\rm Gal}}(\tilde{L}/L)$ and ${{\rm Gal}}(\tilde{J}/J)$ are isomorphic to $\mathfrak A_{n-1}$. Thus ${{\rm Gal}}(\tilde{L}\tilde{J}/\mathbb{Q})\cong \mathfrak A_{n}\times \mathfrak A_{n}$ and ${{\rm Gal}}(\tilde{L}\tilde{J}/LJ)\cong \mathfrak A_{n-1}\times \mathfrak A_{n-1}$. Since $\mathfrak A_{n}$ is simple non-abelian for $n\geq 5$. Thus the only proper nontrivial normal subgroups of $\mathfrak A_{n}\times \mathfrak A_{n}$ are $1\times \mathfrak A_{n}$ and $\mathfrak A_{n}\times 1$ both of which do not contain $\mathfrak A_{n-1}\times \mathfrak A_{n-1}$. Hence by Proposition \ref{simple prop} we are done.

\begin{remark}
    Observe that if $L/K$ is Galois then by Proposition \ref{SGM prop} (5) and Corollary 4.1.8 in \cite{Bhagwat_2025}, $L/K$ is primitive if and only if $L/K$ is general primitive. 
\end{remark}
    
\end{remark}

The following follows from Proposition \ref{SGM prop} (2).

\begin{proposition}\label{gen prim prop}
    Consider $L/K$. Let $G={{\rm Gal}}(\tilde{L}/K)$. If $G$ has less than two proper nontrivial normal subgroups then $L/K$ is general primitive. In particular if $G$ is simple then $L/K$ is general primitive.
\end{proposition}

\subsection{General Primitive Inverse Cluster Size Problem}\hfill
\medskip

We state the following problem.\smallskip

\textbf{Problem}: Let $K$ be a number field. Let $n>2$ and $r|n$ and $r<n$. Does there exist a general primitive extension $L/K$ of degree $n$ with cluster size $r_K(L)=r$?\smallskip

We will denote the above as $(K,n,r)$.    

\begin{remark}
    For the Galois case $(K,n,n)$ we have the same conclusions as Remark \ref{Galois case remark} due to Proposition \ref{SGM prop} (5). 
\end{remark}

\begin{remark}
   Note that Theorem \ref{Prim inv clus} was the Primitive variant of the problem for the cases $(K,n,r)$ where $r<n$ which was solved by the approach of Theorem 3.1.1 in \cite{Bhagwat_2025} in the light of Theorem \ref{unique chain and primitive}. But the General variant of the problem cannot be easily tackled with the same approach because of the conclusions of Theorem \ref{uniq chain coinc GM}. So in this article we take a case by case approach for the General variant and establish some interesting cases. We begin with some lemmas.
\end{remark}

\begin{lemma}[Final Proposition , \cite{perlisroots}] \label{perm} Let $G$ be a transitive subgroup of ${\mathfrak S}_n$ for some $n$. If there exists a finite Galois extension of a field $K$ with Galois group isomorphic to $G$, then there exists an irreducible polynomial $f$ over $K$ of degree $n$
and a labelling of the roots of $f$ so that the Galois group of $f$, viewed as a group permuting roots of $f$, is precisely $G$.
\end{lemma}

\begin{lemma}[Lemma 6.6 (2), \cite{jaiswal2025rootcapacityintersectionindicium}] \label{lemma 6.6}
 Let $K$ be a number field. Given $L/K$, for any solvable group $G$, we have a Galois extension $L'/K$ such that ${{\rm Gal}}(L'/K)=G$ and $L'\cap L=K$.
\end{lemma}

We have a positive answer to the above problem in the following cases.\medskip

\textbf{Case 1}: $(K,\ ^nP_k, k! )$ where $K$ is any number field and $n>2$ and $1 \leq k \leq n-2$. In particular $(K, n, 1)$ where $K$ is any number field and $n>2$. 

\begin{proof}
     By results in \cite{volklein1996groups} on hilbertian fields, we have ${\mathfrak S}_{n}$ to be realisable as a Galois group over $K$. Thus by Lemma \ref{perm}, there exist polynomial $f$ over $K$ of degree $n$ with Galois group ${\mathfrak S}_{n}$. Suppose the roots of $f$ are $\alpha_i\in \bar{K}$ for $1\leq i\leq n$. For $1 \leq k \leq n-2$, let $L_k=K(\alpha_1,\dots, \alpha_k)$. From Theorem 3 in \cite{krithika2023root}, we have $[L_k:K]=\ ^nP_k$ and $r_K(L_k)=k!$. \smallskip

     Now Galois closure for each $L_k/K$ for $1\leq k\leq n-2$ is $L_{n-1}$ with ${{\rm Gal}}(L_{n-1}/K)$ as $\mathfrak S_n$. For $n\neq 4$, there is only one proper nontrivial normal subgroup of $\mathfrak S_n$. Hence we are done by Proposition \ref{gen prim prop}. For $n=4$, there are two proper nontrivial normal subgroups of $\mathfrak S_4$ but one is contained in the other hence $L_k/K$ is general primitive by Proposition \ref{SGM prop} (2).\end{proof}


\textbf{Case 2}: $K$ is any number field.

\begin{enumerate}

\item $(K, 2n, 2)$ where $n\geq 2$.

\item $(K, 4 {n\choose k}, 4)$ where $n>2$ and $1<k<n-1$.
\end{enumerate}

\begin{proof} 

First assume $n>2$. We have $G={\mathfrak S}_{n}$ to be realisable as a Galois group over $K$. For $1\leq k\leq n-1$, let $L/K$ be the fixed field of subgroup $H=\mathfrak{A}_{k}\times \mathfrak{A}_{n-k}$. Now $\mathfrak{A}_{k}\times \mathfrak{A}_{n-k}\unlhd \mathfrak{S}_{k}\times \mathfrak{S}_{n-k}$ and $\mathfrak{S}_{k}\times \mathfrak{S}_{n-k}$ is a maximal subgroup of $\mathfrak{S}_{n}$. Thus $N_{G}(H)=\mathfrak{S}_{k}\times \mathfrak{S}_{n-k}$. 

\begin{enumerate}

\item For $k=1$ or $n-1$, $[L:K]=[\mathfrak{S}_{n}:\mathfrak{A}_{n-1}]=2n$ and $r_K(L)=[N_{G}(H):H]=2$. 

\item  For $1<k<n-1$, $[L:K]=[\mathfrak{S}_{n}:\mathfrak{A}_{k}\times \mathfrak{A}_{n-k}]=4 {n\choose k}$ and $r_K(L)=4$.

\end{enumerate}

Now the case $(K,4,2)$. The dihedral group $D_4$ of order $8$ is solvable hence is realizable over any number field by Lemma \ref{lemma 6.6}. Thus by the Lemma \ref{perm}, there exists $L/K$ with degree $4$ with Galois group of Galois closure as $D_4$. It is easy to see that $r_K(L)=2$. Now since intersection of any two nontrivial normal subgroups of $D_4$ is nontrivial (as the center of $D_4$, a subgroup of order $2$, is contained there), $L/K$ is general primitive by Proposition \ref{SGM prop} (2).\end{proof}

\smallskip

\textbf{Case 3}: $(\mathbb{Q},(p+1)r,r)$ where $p$ is a prime and $2r\ |\ p-1$ and $r\geq 3$.

\begin{proof}

Since $r\geq3$, we can assume $p\geq 5$. By Theorem 1.4 in \cite{zywina2023modular}, $PSL_2(\mathbb{F}_p)$ is realizable as a Galois group over $\mathbb{Q}$ for all primes $p\geq 5$. For any prime $p\geq 5$, let $G=PSL_2(\mathbb{F}_p)$. Now $|G| = (p-1)p(p+1)/{2}$.\smallskip

Now consider the case $r=(p-1)/2$. We know that $C_p\rtimes C_{(p-1)/2}$ is a maximal subgroup of $PSL_2(\mathbb{F}_p)$. Let $H=C_p$. Thus $N_G(H)=C_p\rtimes C_{(p-1)/2}$. Thus $[N_G(H):H]=r$ and $[G:H]=(p+1)(p-1)/2=(p+1)r$.\smallskip

Now let $2r\ |\ p-1$ and $p-1 > 2r$. Let $p-1=rk$ and $c\in \mathbb{F}_p^{\times}$ be an order $k$ element. Consider the following subgroup of $SL_2(\mathbb{F}_p)$.
\[
H =
\left\{
\begin{pmatrix}
  c^l & b \\[4pt]
  0 & c^{-l}
\end{pmatrix}
\; : \; 0\leq l\leq k-1\  ,\  b \in \mathbb{F}_p
\right\}.
\]

The normalizer of $ H $ in $SL_2(\mathbb{F}_p)$ is the following subgroup since it is a maximal subgroup of $SL_2(\mathbb{F}_p)$.
\[ \left\{
\begin{pmatrix}
  a & b \\[4pt]
  0 & a^{-1}
\end{pmatrix}
\; : \;
a \in \mathbb{F}_p^{\times},\; b \in \mathbb{F}_p
\right\}.
\]
\smallskip

Now since $2\mid k$, so $\{\pm I\}\subset H$. Clearly $\{\pm I\}\unlhd SL_2(\mathbb{F}_p)$. Let $H'=H/\{\pm I\}\subset PSL_2(\mathbb{F}_p)=G$. Clearly $N_G(H')=N_{SL_2(\mathbb{F}_p)}(H)/\{\pm I\}$. Hence $[N_G(H'):H']=r$ and $[G:H']=r(p+1)$. Since $PSL_2(\mathbb{F}_p)$ is simple, we are done by Proposition \ref{gen prim prop}.
\end{proof}

\smallskip

\textbf{Case 4}:
 $K$ is any number field and $p$ is a prime and $r\in \mathbb{N}$ such that $r\ |\ p-1$ and $p-1>2r$. 

\begin{enumerate}
    \item $(K,pr,r)$ where $p\equiv 1 \pmod 4 $.

    \item  $(K,pr,r)$ where $p\equiv 3 \pmod 4$ and $r$ is odd.
\end{enumerate}

\begin{proof}

    Consider the following subgroup of $SL_2(\mathbb{F}_p)$. \[
G=B_2(\mathbb{F}_p)=
\left\{
\begin{pmatrix}
  a & b \\[4pt]
  0 & a^{-1}
\end{pmatrix}
\; : \;
a \in \mathbb{F}_p^{\times},\; b \in \mathbb{F}_p
\right\}
\cong \mathbb{F}_p \rtimes \mathbb{F}_p^{\times} \]

The order of $G$ is $p(p-1)$. Let $p-1=rk$ and $c\in \mathbb{F}_p^{\times}$ be an order $k$ element.

\[H=\left<
\begin{pmatrix}
  c & 0 \\[4pt]
  0 & c^{-1}
\end{pmatrix}
\right>
\cong \ <(0,c)> \] 
\smallskip

Consider the normalizer of $H$ in $B_2(\mathbb{F}_p)$ which is $N_{B_2(\mathbb{F}_p)}(H)\cong \{0\}\rtimes \mathbb{F}_p^{\times}$ since $k=(p-1)/r>2$. So $[N_{G}(H):H]=r$ and $[G:H]=pr$. Since $G$ is solvable hence is realizable over any number field by Lemma \ref{lemma 6.6}. Thus we have existence of $L/K$ (precisely the fixed field of $H$) with given degree and cluster size. We also have $\cap_{g\in G}\  gHg^{-1}=\{1\}$ as $k>2$. Thus the Galois group of Galois closure of $L/K$ is $G$ itself. \smallskip

We claim that $L/K$ is general primitive. Suppose not, then it is obtained by a nontrivial strong general magnification from a subextension over $K$. By Proposition \ref{SGM prop}, $G$ is a direct product of nontrivial subgroups. Thus $\mathbb{F}_p \rtimes \mathbb{F}_p^{\times}=AB$ where $A$ and $B$ are nontrivial normal subgroups of $\mathbb{F}_p \rtimes \mathbb{F}_p^{\times}$ such that $A\cap B={1}$. Since $|AB|=p(p-1)=|A||B|$, we can assume without loss of generality that $p\mid |A|$ and $p\nmid |B|$. Thus $\mathbb{F}_p\rtimes \{1\}\subset A$. Since $A$ is a normal subgroup of $\mathbb{F}_p \rtimes \mathbb{F}_p^{\times}$. Hence $A=\mathbb{F}_p\rtimes C$ where $C$ is a subgroup of $\mathbb{F}_p^{\times}$.\smallskip

We also have that elements of $B$ centralize elements of $A$ and that $p\nmid |B|$ and $A\cap B=1$. Thus one can show that $B=\ <(0,-1)>\ \cong\  <-I>$. Thus $|C|=(p-1)/2$. We know that unique index-$2$ subgroup of $\mathbb{F}_p^{\times}$ is $(\mathbb{F}_p^{\times})^2$. Thus $C=(\mathbb{F}_p^{\times})^2$. Now consider the case $p\equiv 1\pmod 4$. We have that $-1\in C$. Thus $(0,-1)\in A$ which contradicts $A\cap B=1$. \smallskip

Now consider the case $p\equiv 3\pmod 4$ and $r$ is odd. Thus $-1\not \in C$. Thus we have $\mathbb{F}_p \rtimes \mathbb{F}_p^{\times}$ is direct product of subgroups $\mathbb{F}_p\rtimes (\mathbb{F}_p^{\times})^2$ and $<(0,-1)>$. Also $\mathbb{F}_p^{\times}$ is direct product of $(\mathbb{F}_p^{\times})^2$ and $<-1>$. Now since $(p-1)/2$ is odd and $r$ is odd, we have $k=(p-1)/r$ to be even and $k/2$ to be odd. Thus $<c>$ is direct product of subgroups $<c^2>$ and $<c^{k/2}>\ =\ <-1>$. Thus $H$ is a direct product of subgroups $<(0,c^2)>\ \subseteq \mathbb{F}_p\rtimes (\mathbb{F}_p^{\times})^2$ and $<(0,-1)>\ \subseteq \ <(0,-1)>$. Thus by Proposition \ref{SGM criterion}, we have a contradiction to the nontriviality.\end{proof}

\begin{remark} A similar argument fails for the case $(K,pr,r)$ where $p\equiv 3\pmod 4$ and $r$ is even. In that case $k$ is odd. Thus $H$ is a direct product of subgroups $<(0,c)>\ \subseteq \mathbb{F}_p\rtimes (\mathbb{F}_p^{\times})^2$ and $\{(0,1)\} \ \subseteq \ <(0,-1)>$. By Proposition \ref{SGM criterion}, $L/K$ is indeed obtained by a nontrivial strong general magnification from a subextension over $K$. In fact it is obtained by a nontrivial strong cluster magnification from a subextension over $K$.
\end{remark}

\smallskip

\textbf{Case 5}: $(K,2n,n)$ where $K$ is any number field and $n>2$ is odd.


\begin{proof}

Consider the dihedral group of order $2n$ and cyclic group of order $n$.\[
D_n = \langle e,f \mid e^n=f^2=1,\ fef=e^{-1}\rangle,
\quad
C_n=\langle g\rangle .
\]

Consider the cyclic subgroup $H=\langle (e,g)\rangle \subset D_n \times C_n=G$. Then $|H|=n$. By computing, \[
N_{G}(H)=\{(e^i,g^j)\mid 0\le i,j < n\}= C_n\times C_n .
\]

Hence $[N_{G}(H):H] = n$ and $[G:H] = 2n$. Since $G=D_n\times C_n$ is solvable, it is realizable over any number field by Lemma \ref{lemma 6.6}. Thus we have existence of $L/K$ (precisely the fixed field of $H$) with given degree and cluster size. We also have $\cap_{g\in G}\  gHg^{-1}=\{1\}$. Thus the Galois group of Galois closure of $L/K$ is $G$ itself. \smallskip

We claim that $L/K$ is general primitive. Suppose not, then it is obtained by a nontrivial strong general magnification from a subextension over $K$. By Proposition \ref{SGM prop}, $G$ is a direct product of nontrivial subgroups. Thus $D_n\times C_n=AB$ where $A$ and $B$ are nontrivial normal subgroups of $D_n\times C_n$ such that $A\cap B={1}$. Also by the same proposition, $H$ is a direct product of subgroups $H=A'B'$ such that $A'\subset A$ and $B'\subset B$ with $[A:A']>1$ and $[B:B']>1$. Since $|H|=n$, we can assume that $|A'|=a$ and $|B'|=b$ where $n=ab$ and $gcd(a,b)=1$. Since $H$ is cyclic we have $A'=\langle (e^b,g^b) \rangle$ and $B'=\langle (e^a,g^a) \rangle$.

\smallskip

Now since $A'\subset A$, we have $B\subset C_G(A')$ where $C_G(A')$ denotes the centralizer of $A'$ in $G$. Similarly $A\subset C_G(B')$. Suppose $a,b>1$. Then by computing, $C_G(A')=C_G(B')=C_n\times C_n$ as $n$ is odd. This gives $AB\subset C_n\times C_n$ which is a contradiction. Thus one of $a$ and $b$ has to be $1$. Suppose without loss of generality $a=n$ and $b=1$. Then $A'=H$ and $B'=1$. Since $A\cap B=1$, so $H\cap B=1$. Also $B\subset C_n\times C_n$ is a nontrivial subgroup. Hence $(e^i,g^j) \in B$ for some $0\leq i,j<n$ with $i\neq j$. \smallskip

Since $|B|\mid n^2$ and $n$ is odd, we have $2\mid |A|$. By Cauchy's theorem there is an order $2$ element in $A$. Thus as $n$ is odd, we have $(e^kf,1)\in A$ for some $0\leq k< n$. Now $A\subset C_G(B)$. So $(e^kf,1)(e^i,g^j)=(e^i,g^j)(e^kf,1)$. Thus $e^kfe^i=e^ie^kf$. Thus $e^{-i}=e^i$. Since $n$ is odd, $i=0$. Thus $(1,g^j) \in B$ for some $0<j<n$. Note that $(1,g^{2j})\in B$ which is non-identity as $n$ is odd. Now consider the following element in $A$ which contradicts $A\cap B=1$. \[(e^kf,1)(e^j,g^j)(e^kf,1)(e^j,g^j)=(e^kfe^{k+j}fe^j, g^{2j})=(fe^{-k}e^{k+j}e^{-j}f,g^{2j})=(f^2,g^{2j})=(1,g^{2j}).\]

\end{proof}

\section{Totally Real Inverse Cluster Size Problem}\label{tot real}

Recall the following.
\begin{definition}
    A number field is called totally real if for each embedding of the number field into $\mathbb{C}$, the image lies inside $\mathbb{R}$.
\end{definition}

\begin{proposition}\label{comp}
   Subfields of a totally real number field are totally real. Compositum of totally real number fields is totally real.
\end{proposition}

\begin{corollary}\label{tot real cor}
    Let $K$ be a totally real number field. Consider an extension $M/K$ and its Galois closure $\tilde{M}/K$. Then $M$ is totally real if and only if $\tilde{M}$ is totally real.
\end{corollary}

\begin{proof}

Now $M$ is a subfield of $\tilde{M}$ and $\tilde{M}$ is compositum of conjugates of $M$ over $K$ (compositum of images of embeddings of $M$ into $\mathbb{C}$ fixing $K$). By Proposition \ref{comp} we are done.
\end{proof}

\subsection{Totally Real Inverse Cluster Size Problem}\label{tot real subsec}\hfill \medskip

We establish the following completely in this section.

\begin{theorem}\label{Tot real var} (Totally Real Inverse Cluster Size Problem) Let $K$ be a totally real number field and $n>2$ and $r\mid n$. There exists an $L/K$ where $L$ is totally real such that $[L:K]=n$ and $r_K(L) = r$. 
\end{theorem}


















\smallskip

This theorem will be proved as a consequence of the following two lemmas. We begin with modifying Lemma 2 in \cite{krithika2023root} for our purpose.

  \begin{lemma}\label{vanchi lemma}
      Let $K$ be a totally real number field and consider any finite extension $L/K$. Let $d\geq 2$. Then there exist infinitely many cyclic Galois extensions $F/K$ of degree $d$ where $F$ is totally real such that $L$ and $F$ are linearly disjoint over $K$.
  \end{lemma}

\begin{proof}
    Case 1 : $K=\Q$ :
	Let $\Delta_L\in\Z$ be the discriminant of the number field
$L$. For a prime $p$ not dividing $\Delta_L$, consider the 
cyclotomic extension $M=\Q(\zeta_p)$ where $\zeta_p$ is primitive $p$-th root  of unity. Now $p$ is the only prime ramified in $M$ but it remains unramified in $L$. So $L\cap M=\Q$. Now $M'=\mathbb{Q}(\zeta_p+\zeta_p^{-1})$ is a totally real number field inside $M$ with $[M':\mathbb{Q}]=(p-1)/2$. We also have that $L\cap M'=\Q$. So $L$ and $M'$ are linearly disjoint over $\mathbb{Q}$ as $M'/\mathbb{Q}$ is Galois (in fact cyclic).\smallskip

Now for a given $d\geq 2$, we can find infinitely many primes $p$ such that $p>\Delta_L$ and $p\equiv 1 \pmod {2d}$ by Dirichlet's theorem on  
arithmetic progressions.
For such $p$ the extensions $M'/\mathbb{Q}$ will have a cyclic Galois subextension $F/\mathbb{Q}$ of degree $d$. The field $F$ is totally real by Proposition \ref{comp} and $L$ and $F$ are linearly disjoint over $\Q$.
\smallskip

Case 2 : For any totally real number field $K$ :
	For given  extension $L/K$, we get extensions $F_1/\Q$ by Case 1 such that $L$ and $F_1$ are linearly disjoint over $\Q$ and $F_1$ is totally real. Thus the extensions $L$ and $F=K F_1$ will be  linearly disjoint over $K$. Also $F=K F_1$  is totally real by Proposition \ref{comp} and $F/K$ is Galois and has the same degree as $F_1/\Q$.\end{proof}

\begin{lemma}\label{Sn}


Let $K$ be a totally real number field. 

\begin{enumerate}
\item For every $n\geq 2$, there exists a Galois $L/K$ where $L$ is totally real such that ${{\rm Gal}}(L/K)\cong \mathfrak S_n$.\smallskip

\item 

 There exists a degree $4$ extension $L/K$ such that $L$ is totally real and ${{\rm Gal}}(\tilde{L}/K)\cong D_4$ where $D_4$ is the dihedral group of order 8.  \smallskip

 \item  
For any odd prime $p$, there exists an extension $L/K$ such that $L$ is totally real and ${{\rm Gal}}(L/K) \cong D_p \times C_p$.

\end{enumerate}

\end{lemma}

\begin{proof}
\hfill
\begin{enumerate}
\item Let $\Delta_K\in\Z$ be the discriminant of the totally real number field
$K$. By Theorem 1.1 in \cite{kedlaya2012construction}, we have infinitely many monic irreducible polynomials $f(x)\in \mathbb{Q}[x]$ of degree $n$ with integer coefficients having all $n$ roots to be real such that the discriminant of $f(x)$ i.e. $\Delta_f$ is coprime to $\Delta_K$ and ${{\rm Gal}}(\mathbb{Q}_f/\mathbb{Q})\cong \mathfrak S_n$ where $\mathbb{Q}_f$ is splitting field of $f$ over $\mathbb{Q}$. Let $M=\mathbb{Q}(\alpha)$ where $\alpha$ is a root of $f$. Thus $M$ is totally real and ${{\rm Gal}}(\tilde{M}/\mathbb{Q})\cong \mathfrak S_n$. Also $\tilde{M}$ is totally real by Corollary \ref{tot real cor}. \smallskip

We have the well known relation $\Delta_f = \Delta_{M}\cdot [\mathcal{O}_M : \mathbb{Z}[\alpha]]^2$ where $\mathcal{O}_M$ is the ring of integers of $M$ and $\Delta_M\in\Z$ is the discriminant of the totally real number field $M$. Thus $\Delta_M$ and $\Delta_K$ are coprime. Let $\Delta_{\tilde{M}}\in\Z$ be the discriminant of the totally real number field $\tilde{M}$. We know that a prime ramifies in $M/\mathbb{Q}$ if and only if it ramifies in $\tilde{M}/\mathbb{Q}$. Thus a prime $p\mid \Delta_M\iff p\mid \Delta_{\tilde{M}}$. Hence $\Delta_{\tilde{M}}$ and $\Delta_K$ are also coprime. Hence $K\cap \tilde{M}=\mathbb{Q}$. Let $L=K\tilde{M}$. So $L$ is totally real by Proposition \ref{comp}. Also ${{\rm Gal}}(L/K)={{\rm Gal}}(K\tilde{M}/K)\cong {{\rm Gal}}(\tilde{M}/\mathbb{Q})\cong \mathfrak S_n$. \smallskip

\item  By Theorem 3.2 in \cite{vanchinathan2023exceptional}, we have a family of irreducible degree $4$ polynomials $f_t(x)\in \mathbb{Q}[x]$ for integers $t\geq 7$ given by \[f_t(x)=x^4-tx^3-3x^2+tx+1\] such that the splitting field has Galois group $D_4$. By Lemma 2.2 in \cite{vanchinathan2023exceptional}, \[\Delta_{f_t}=(t^2+4)^2 (4t^2+25)\] and all $4$ roots of $f_t$ are real. Observe that a prime $p$ divides $\Delta_{f_t}$ if and only if $t^2\equiv -4 \pmod p$ or $4t^2\equiv -25 \pmod p$. Thus there are at most $4$ residue classes modulo $p$ for $t$ such that $p\mid \Delta_{f_t}$. Hence if $p>4$ then there exists a residue class modulo $p$ for $t$ such that $p\nmid\Delta_{f_t}$. Also for the cases $p=2$ and $p=3$ its easily seen that  the residue classes $1\pmod 2$ and $1\pmod 3$ respectively for $t$ would ensure that $p\nmid \Delta_{f_t}$.\smallskip 
 
 Now consider the primes dividing the discriminant $\Delta_K$ of $K/\mathbb{Q}$ and choose a residue class for $t$ for each of these primes which ensures $p\nmid \Delta_{f_t}$. Thus by CRT, there are infinitely many integers $t\geq 7$ such that $p\nmid \Delta_{f_t}$ for all primes $p$ dividing $\Delta_K$. Thus we have $\Delta_{f_t}$ and $\Delta_K$ to be coprime. Now by a very similar argument as in part (1), we have existence of a degree $4$ extension $L/K$ such that $L$ is totally real and ${{\rm Gal}}(\tilde{L}/K)\cong D_4$. 

 \smallskip
 Under the assumption that discriminant $\Delta_K$ is odd, we can actually give an alternate proof for the above. Let $p_1,p_2\dots, p_m$ be all the distinct odd primes dividing $\Delta_K$. We have infinitely many primes $p$ such that $p>3$ and $p\equiv 3\pmod {4p_1p_2\cdots p_m}$ by Dirichlet's theorem on arithmetic progressions. Since $p\equiv 3\pmod 4$  we have that $p-1$ is not a square in $\mathbb{Z}$. Thus by applying Corollary 4.5 and Example 4.6 in \cite{conrad2010galois}, we have that the extension $M/\mathbb{Q}$, where $M=\mathbb{Q}(\sqrt{p+\sqrt{p}})$, is a degree $4$ extension and ${{\rm Gal}}(\tilde{M}/\mathbb{Q})\cong D_4$. Clearly $M$ is also totally real. Let $f(x)$ be the minimal polynomial of $\sqrt{p+\sqrt{p}}$ over $\mathbb{Q}$. Then the discriminant of $f(x)$ i.e. $\Delta_f$ is $256p^3(p-1)$. Due to the clever choice of $p$ we have $\Delta_f$ and $\Delta_K$ to be coprime.

 \smallskip

 \item  Since $\tilde{K}/\mathbb{Q}$ (Galois closure of $K/\mathbb{Q}$) is a finite separable extension, it contains only finitely many quadratic subextensions. Choose a real quadratic field $F/\mathbb{Q}$ that is not contained in $\tilde{K}$. By Theorem I.2.1 in \cite{jensen1982polynomials}, there exist infinitely many extensions $M/\mathbb{Q}$ containing $F$ with ${{\rm Gal}}(M/\mathbb{Q})\cong D_p$. By Remark II.2.6 in \cite{jensen1982polynomials}, since $F$ is real quadratic, we have $M$ to be totally real.\smallskip

  Now as $p$ is an odd prime, the only nontrivial normal subgroup of $D_p$ is $C_p$ and we have ${{\rm Gal}}(M/F)\cong C_p$. Thus the only nontrivial Galois subextension of $M/\mathbb{Q}$ is $F/\mathbb{Q}$. Also note that $(M\cap \tilde{K})/\mathbb{Q}$ is Galois. Thus $(M\cap \tilde{K})/\mathbb{Q}$ is either $M$ or $F$ or $\mathbb{Q}$. As $\tilde{K}$ does not contain $F$, we have $M\cap\tilde{K}=\mathbb{Q}$. Thus $M\cap K=\mathbb{Q}$. Let $L'=KM$ which is totally real by Proposition \ref{comp}. Also ${{\rm Gal}}(L'/K)\cong {{\rm Gal}}(M/\mathbb{Q}) \cong D_p$.\smallskip

Now by Lemma \ref{vanchi lemma}, there exist infinitely many cyclic Galois extensions $F'/K$ of degree $p$ where $F'$ is totally real such that $L'$ and $F'$ are linearly disjoint over $K$. Let $L=L'F'$ which is totally real. Also ${{\rm Gal}}(L/K)\cong {{\rm Gal}} (L'/K)\times {{\rm Gal}}(F'/K)\cong D_p\times C_p$.\end{enumerate}\end{proof}

         

\begin{proof}[Proof of Theorem \ref{Tot real var}]\hfill

\begin{enumerate}
\item 
The case $n=r$ follows from Lemma \ref{vanchi lemma} applied for $K/K$ and $d=n$.\smallskip

\item  Now let $r\mid n$ and $n>2r$. So $n/r >2$. We get an extension $L'/K$ with $L'$ being totally real and $[L':K]=n/r$ and ${{\rm Gal}}(\tilde{L'}/K)\cong \mathfrak S_{n/r}$ similarly as in Theorem \ref{Prim inv clus} by applying Lemma \ref{Sn} (1). Thus $r_{K}(L')=1$. This in particular solves the case $r=1$ and $n>2$.\smallskip


Now let $r\geq 2$. By Lemma \ref{vanchi lemma} applied to $\tilde{L'}/K$, we have a Galois extension $J/\mathbb{Q}$ of degree $r$ where $J$ is totally real such that $\tilde{L'}$ and $J$ are linearly disjoint over $K$. 
Now let $L=L'J$. Thus $L/K$ is obtained by strong general magnification from $L'/K$ through $J/K$ (in fact it is obtained by strong cluster magnification). Therefore by Proposition \ref{GM}, we have that $[L:K]=[L':K]\cdot [J:K]=(n/r)\cdot r=n$ and $r_{K}(L)=r_{K}(L')\cdot r_{K}(J)=1\cdot r= r$. Also $L=L'J$ is totally real by Proposition \ref{comp}.\smallskip

\item Now consider the case $n=2r$ and $r\neq 2^m$ for any $m\geq 1$. Thus there is an odd prime $p$ such that $p\mid r$. By Lemma \ref{Sn} (3), there exists an extension $M/K$ such that $M$ is totally real and ${{\rm Gal}}(M/K) \cong D_p \times C_p$. By proof of Case 5 in Section \ref{general primitive}, we have existence of a subextension $L'/K$ of $M/K$ with $[L':K]=2p$ and $r_K(L')=p$. Since $M$ is totally real, so is $L'$ by Proposition \ref{comp}. This in particular solves the case $n=2p,\ r=p$.\smallskip

Now let $r/p \geq 2$. By Lemma \ref{vanchi lemma} applied to $\tilde{L'}/K$, we have a Galois extension $J/\mathbb{Q}$ of degree $r/p$ where $J$ is totally real such that $\tilde{L'}$ and $J$ are linearly disjoint over $K$. Now let $L=L'J$ which is totally real. By Proposition \ref{GM}, $[L:K]=[L':K]\cdot [J:K]=(2p)\cdot (r/p)=2r=n$ and $r_{K}(L)=r_{K}(L')\cdot r_{K}(J)=p\cdot (r/p)= r$.

\smallskip

\item Now consider the case where $n=2r$ and $r=2^m$ for some $m\geq 1$. The special case where $n=4, r=2$ is done by applying Lemma \ref{Sn} (2). The general case is achieved by applying Lemma \ref{vanchi lemma} to the special case (similar to previous paragraph).\end{enumerate}\end{proof}

\begin{remark}
    In the above, the obtained extension $L/K$ is not primitive (hence not general primitive also) for the cases (i) $n\neq 2r$ with $1<r<n$ and (ii) $n=2r$ where $r$ is not prime.
\end{remark}

\subsection{Totally Real General Primitive Inverse Cluster Size Problem}\hfill

We can show existence of general primitive $L/K$ (hence also primitive) where $L$ is totally real number field for the following cases of $(K,n,r)$.

\begin{enumerate}
\item $(K,\ ^nP_k, k! )$ where $n>2$ and $1 \leq k \leq n-2$. In particular $(K, n, 1)$ where $n>2$.\smallskip

\item $(K, 2n, 2)$ where $n\geq 2$.
\smallskip


\smallskip

\item $(K, 4 {n\choose k}, 4)$ where $n>2$ and $1<k<n-1$.

\smallskip

\item $(K,2p,p)$ where $p>2$ is an odd prime.

\end{enumerate}

\begin{proof}
Follows from Lemma \ref{Sn} and proofs of Cases (1), (2) and (5) in Section \ref{general primitive}.     
\end{proof}

\subsection{Totally Real Primitive Inverse Cluster Size Problem}\hfill \medskip


We show existence of primitive $L/K$ where $L$ is totally real number field for the following case. \smallskip

\textbf{Case}: $(K,n,r)$ where $K$ is any totally real number field and $n$ is odd and $1<r<n$.

\begin{lemma}\label{odd solv lemma}
Let $K$ be a totally real number field and let $G$ be a group of odd order. Then there exists a Galois extension $M/K$ where $M$ is totally real such that ${{\rm Gal}}(M/K)\cong G$.
\end{lemma}

\begin{proof}
We know by Feit-Thompson theorem that $G$ is solvable as $G$ has odd order. Therefore by Lemma \ref{lemma 6.6} applied to the extension $K/\mathbb{Q}$, we have a Galois extension $L'/\mathbb{Q}$ such that ${{\rm Gal}}(L'/\mathbb{Q})=G$ and $L'\cap K=\mathbb{Q}$.\smallskip

The proof of $L'$ being totally real is essentially the proof of Proposition 5.2.2 in \cite{kalyanswamy2012inverse} but we provide it nonetheless. As $L'/\mathbb{Q}$ is Galois, then when regarding $L'$ as a subfield of $\mathbb{C}$, embeddings of $L'$ are given by elements of ${{\rm Gal}}(L'/\mathbb{Q})$. If $L'$ is not totally real then complex conjugation would be an embedding of $L'$ into $\mathbb{C}$ and hence an even order element of the Galois group which is itself of odd order and this contradicts Lagrange's theorem.\smallskip

Now let $M=KL'$. As $L'\cap K=\mathbb{Q}$, so ${{\rm Gal}}(M/K)={{\rm Gal}}(KL'/K)\cong {{\rm Gal}}(L'/\mathbb{Q})=G$. Also $M=KL'$ is totally real by Proposition \ref{comp}.\end{proof}

\begin{proof}[Proof of the Case]

Now $r>1$. Consider the solvable group $G=(\Z/r\Z)^s \rtimes \Z / s\Z$ in the proof of Theorem \ref{Prim inv clus}. The action of this group is transitive on $n$ points, and a point stabiliser fixes precisely $r$ points. Any point stabiliser $H$ is isomorphic to $(\Z/r\Z)^{s-1}$. Now $n=rs$ and $|G|=r^s s$. So $n$ is odd if and only if $|G|$ is odd. By applying Lemma \ref{odd solv lemma} for this $G$ when $n$ is odd, we get a Galois extension $M/K$ where $M$ is totally real such that ${{\rm Gal}}(M/K)\cong G$. Let  $L=M^H$. So $[L:K]=n$ and $r_K(L)=r$. Also $L$ is totally real by Proposition \ref{comp} and also $L/K$ is primitive by proof of Theorem \ref{Prim inv clus}.\end{proof}






\section{Conflict of interest statement}
On behalf of all authors, the corresponding author states that there is no conflict of interest.

\section{Data availability statement}

\emph{Our manuscript has no associated data.}

\bigskip

\noindent {\it Acknowledgements:} The authors would like to thank Chennai Mathematical Institute and Institute of Mathematical Sciences, Chennai for providing an excellent environment where a part of this research work was carried out. The first author would also like to acknowlege P Narayanan for suggesting to consider the totally real variant of the problem. \medskip

\bibliographystyle{plain}
 \bibliography{mybib}

\begin{thebibliography}{10}

\bibitem{Bhagwat_2025}
Chandrasheel Bhagwat and Shubham Jaiswal.
\newblock Cluster magnification, root capacity, unique chains, base change and ascending index.
\newblock {\em Proc. Indian Acad. Sci. (Math. Sci.) 135:19}, 2025.

\bibitem{vanchinathan2023exceptional}
Aruna C and P~Vanchinathan.
\newblock Exceptional quartics are ubiquitous.
\newblock {\em Accepted for publication in Proc. Indian Acad. Sci. (Math. Sci.), https://arxiv.org/abs/2306.17556}, 2023.

\bibitem{conrad2010galois}
Keith Conrad.
\newblock {Galois groups of cubics and quartics (not in characteristic 2)}.
\newblock {\em https://kconrad.math.uconn.edu/blurbs/galoistheory/cubicquartic.pdf}, 2010.

\bibitem{conrad2023galois}
Keith Conrad.
\newblock The {Galois} correspondence at work.
\newblock {\em \\ https://kconrad.math.uconn.edu/blurbs/galoistheory/galoiscorrthms.pdf}, 2023.

\bibitem{jaiswal2025rootcapacityintersectionindicium}
Shubham Jaiswal.
\newblock {On Root Capacity, Intersection Indicium, Minimal Generating Sets of Galois Closure \& Compositum Feasible Triplets. arXiv preprint}.
\newblock {\em https://arxiv.org/abs/2505.20238}, 2025.

\bibitem{jensen1982polynomials}
Christian~U Jensen and Noriko Yui.
\newblock {Polynomials with Dp as Galois group}.
\newblock {\em Journal of Number Theory}, 15(3):347--375, 1982.

\bibitem{kalyanswamy2012inverse}
Sudesh Kalyanswamy.
\newblock {Inverse Galois problem for totally real number fields}.
\newblock {\em Cornell University Mathematics Department Senior Thesis}, 2012.

\bibitem{kedlaya2012construction}
Kiran Kedlaya.
\newblock A construction of polynomials with squarefree discriminants.
\newblock {\em Proceedings of the American Mathematical Society}, 140(9):3025--3033, 2012.

\bibitem{krithika2023root}
M~Krithika and P~Vanchinathan.
\newblock An elementary problem in {Galois} theory about the roots of irreducible polynomials.
\newblock {\em Proc. Indian Acad. Sci. (Math. Sci.) 134:28}, 2024.

\bibitem{krithika2025inflated}
M~Krithika and P~Vanchinathan.
\newblock {Inflated G-Extensions for Algebraic Number Fields. arXiv preprint}.
\newblock {\em Accepted for publication in Czechoslovak Mathematical Journal, https://arxiv.org/abs/2503.23946}, 2025.

\bibitem{perlisroots}
Alexander~R Perlis.
\newblock Roots appear in quanta: exercise solutions.
\newblock {\em \\ https://www.math.lsu.edu/~aperlis/publications/rootsinquanta/}, 2003.

\bibitem{perlis2004roots}
Alexander~R Perlis.
\newblock Roots appear in quanta.
\newblock {\em The American Mathematical Monthly}, 111(1):61--63, 2004.

\bibitem{volklein1996groups}
Helmut V{\"o}lklein.
\newblock {\em Groups as {Galois} groups: an introduction}.
\newblock Number~53. Cambridge University Press, 1996.

\bibitem{zywina2023modular}
David Zywina.
\newblock Modular forms and some cases of the {Inverse Galois Problem}.
\newblock {\em Canadian Mathematical Bulletin}, 66(2):568--586, 2023.

\end{thebibliography}

\end{document}